%% file: Distinguishing_Tournaments.tex
\documentclass[11pt,a4paper]{article}
\usepackage{amsthm,amssymb,amsfonts,amsmath,lineno,mathtools,graphicx,tikz}

\include{graphs}

\usepackage{epsfig,xspace,algorithmic,algorithm,verbatim}
\usepackage{color}

\newtheorem{theorem}{Theorem}
\newtheorem{claim}{Claim}
\newtheorem{proposition}{Proposition}
\newtheorem{corollary}{Corollary}
\newtheorem{definition}{Definition}

\newtheorem{observation}{Observation}

\newtheorem{question}{Question}

\begin{document}

\title{Distinguishing Tournaments with Small Label Classes}

\author{Antoni Lozano\thanks{Research Group on Combinatorics, Graph Theory, and Applications, 
Universitat Polit\`ecnica de Catalunya, Catalonia, {\tt antoni@lsi.upc.edu}.
Research supported by projects TIN2014-57226-P (APCOM) and MTM2014-600127-P.}}

\maketitle


\begin{abstract}\noindent
A {\em $d$-distinguishing vertex (arc) labeling} of a digraph is a vertex (arc) labeling using $d$ labels that is not preserved by any nontrivial automorphism. Let $\rho(T)$ ($\rho'(T)$) be the minimum size of a label class in a 2-distinguishing vertex (arc) labeling of a tournament $T$. Gluck's Theorem~\cite{G} implies that $\rho(T) \le \lfloor n/2 \rfloor$ for any tournament $T$ of order $n$. In this paper we construct a family of tournaments $\cal H$ such that $\rho(T) \ge \lfloor n/2 \rfloor$ for any order $n$ tournament in $\cal H$. Additionally, we prove that $\rho'(T) \le \lfloor 7n/36 \rfloor + 3$ for any tournament $T$ of order $n$ and $\rho'(T) \ge \lceil n/6 \rceil$ when $T \in {\cal H}$ and has order $n$. These results answer some open questions stated by Boutin~\cite{B1, B2}.
\end{abstract} 

\section{Introduction}\label{sec:int}

We follow the standard notation in graph theory. In particular, given a directed graph ({\em digraph} for short) $G$, $V(G)$ ($A(G)$) stands for its set of vertices (arcs) and $Aut(G)$ denotes the automorphism group of $G$. We refer to the identity automorphism in $Aut(G)$ as to the {\em trivial} automorphism. A tournament is a complete oriented graph, that is, a digraph $T$ for which for every $u,v \in V(T)$, either $uv \in A(T)$ or $vu \in A(T)$ but not both.

A {\em vertex (arc) labeling} of a digraph $G$ is a total function $\phi: V(G) \rightarrow L$ ($\phi: A(G) \rightarrow L$) which labels each vertex (arc) of $G$ with a label from the set $L$. Given a vertex labeling $\phi$ for a digraph $G$, we say that an automorphism $\sigma \in Aut(G)$ {\em preserves} $\phi$ if $\phi(\sigma(v)) = \phi(v)$ for every vertex $v \in V(G)$. Similarly, we say that $\sigma \in Aut(G)$ {\em preserves} an arc labeling $\phi$ if $\phi(uv) = \phi(\sigma(u)\sigma(v))$ for every arc $uv \in A(G)$. On the contrary, a vertex or arc labeling $\phi$ {\em breaks} an automorphism $\sigma \in Aut(G)$ if $\phi$ is not preserved by $\sigma$. A (vertex or arc) labeling $\phi$ of $G$ that breaks all nontrivial automorphisms in $Aut(G)$ is called {\em distinguishing} for $G$. Additionally, if $\phi$ uses $d$ labels, it is called $d$-distinguighing for $G$.

Albertson and Collins introduced the concept of {\em distinguishing number} in the seminal paper~\cite{AC1} as the instantiation of the idea of ``symmetry breaking'' in graphs. The {\em distin\-guishing number} $D(G)$ of a digraph $G$ is the least cardinal $d$ such that $G$ has a $d$-distinguishing vertex labeling. In recent years, this concept has been extended to the {\em distinguishing index} $D'(G)$, which is defined as the least cardinal $d$ such that $G$ has an $d$-distinghishing arc labeling. A {\em distinguishing vertex class (distinguishing arc class)} of $\phi$ in $G$ is any of the $d$ subsets of $V(G)$ ($A(G)$) having the same label under $\phi$. These notions have been studied in \cite{B1,B2,B3,BP,C,KPW,P}.

\medskip

With respect to tournaments, Albertson and Collins \cite{AC2} conjectured that every tournament $T$ satisfies $D(T) \le 2$. As Godsil observed in 2002~\cite{Go}, since tournaments have odd order automorphism groups, the conjecture follows from Gluck's Theorem (\cite{G}, see also the shorter and self-contained proof in \cite{M}). In the following statement of Gluck's Theorem, given a permutation group $G$ on $\Omega$, $S \subseteq \Omega$ is a regular subset of $G$ if the setwise stabilizer $\{ g \in G \mid Sg = S\}$ only contains the identity. Therefore, a regular subset plays a similar role to that of a 2-distinguishig vertex class.

\begin{theorem} (Gluck's Theorem, \cite{G,M}). Let G be a permutation group of odd order on a finite set $\Omega$. Then G has a regular subset in $\Omega$.
\end{theorem}

Given a tournament $T$, Gluck's Theorem shows the existence of a regular subset $S \subseteq \Omega = V(T)$ for $Aut(T)$. Define a labeling $\phi$ that assigns label 1 to the vertices in $S$ and label 2 to the vertices in $V(T) \setminus S$. Now, the definition of regular subset implies that the only automorphism in $Aut(T)$ preserving labeling $\phi$ is the identity. Therefore, $\phi$ constitutes a 2-distinguishing vertex labeling of the vertices of $T$ and the following fact can be claimed.

\begin{corollary}\cite{Go}
If $T$ is a tournament, then $D(T) \le 2$.
\end{corollary}

As an added consequence of Gluck's Theorem, we can observe that the distinguishing index of tournaments is also bounded by 2. Suppose $S$ is the regular subset, given by Gluck's Theorem, of the vertices of a tournament $T$. Clearly, vertices in $S$ can be singularized if the arcs lying inside $S$ are labeled with 1 and the rest are labeled with 2. This way, the orbit of a vertex in $S$ by any automorphism will lie inside $S$, and the previous arc labeling will be 2-distinguishing.

\begin{corollary}
If $T$ is a tournament, then $D'(T) \le 2$.
\end{corollary} 

Some literature on the subject has focused on the minimum possible size of a distinguishing vertex class, which has been called {\em the cost of 2-distinguishing}. We define it here both for vertices and arcs. For a digraph $G$ such that $D(G) \le 2$, define $\rho(G)$ ($\rho'(G)$) as the minimum size of a distinguishing vertex (arc) class. 

The article is organized as follows. Section~\ref{sec:pre} defines and studies an infinite class $\cal H$ of tournaments that gives rise to lower bounds for $\rho$ and $\rho'$. In Section~\ref{sec:d_num} we prove  $\rho(T) \le \lfloor n/2 \rfloor$ for any tournament $T$ of order $n$ and show that this bound is exact for tournaments in $\cal H$. In Section~\ref{sec:d_ind} we show $\rho'(T) \le \lfloor 7n/36 \rfloor + 3$ for any tournament $T$ of order $n$ and prove a lower bound of $\lceil n/6 \rceil$ for tournaments of order $n$ in $\cal H$. Finally, some conclusions and open problems are discussed in Section~\ref{sec:con}.

\section{Class $\cal H$ and black and white labelings}\label{sec:pre}

We introduce here a class of tournaments that will be used in the succeeding sections to provide lower bounds for the cost of 2-distinguishing tournaments with vertices ($\rho$) and with arcs ($\rho'$).
 
By $\vec{C}_3$ we denote the directed triangle, that is, the tournament containing the vertices $x_1$, $x_2$, and $x_3$ and the arcs $x_1 x_2$, $x_2 x_3$, and $x_3 x_1$. 

\begin{definition}
The family ${\cal H} = \{ H_k \}_{k \ge 0}$ of tournaments is inductively defined as follows: 
\begin{itemize}
\item $H_0$ is a single vertex tournament and 
\item $H_k$, for $k > 0$, is the tournament consisting of a copy of $\vec{C}_3$ in which every vertex $x_i$ in $\vec{C}_3$ is substituted by a copy of $H_{k-1}$, called {\em tertian} $T_i$, and an arc $x_i x_j \in A(C_3)$ is substituted by all possible arcs from $T_i$ to $T_j$. 
\end{itemize}
\end{definition}

\begin{observation}
For any $k \ge 0$, $|V(H_k)| = 3^k$.
\end{observation}

A {\em module} in a tournament $T$ is a set $X$ of vertices such that each vertex in $V(T) \setminus X$ has a uniform relationship to all vertices in $X$, that is, for every vertex $v \in V(T) \setminus X$, either $uv \in A(T)$ for all $u \in X$ or $vu \in A(T)$ for all $u \in X$. Note that $T$ and sets $\{u\}$, where $u \in V(T)$, are modules. Furthermore, modularity is transitive: if $Y$ is a module in the subtournament $T[X]$ induced by module $X$, then $Y$ is a module in $T$.

According to the definition of $H_k$, each of its three tertians are modules. By transitivity of modularity we can make the following observation.

\begin{observation}\label{ob:disjoint}
For every $k \ge 1$, $H_k$ can be decomposed into $3^{k-1}$ pairwise disjoint modules isomorphic to $\vec{C}_3$.
\end{observation}




We also need the following property on how vertices in $H_k$ can move in an automorphism.

\begin{proposition}\label{pr:module_move}
Let $\sigma \in Aut(H_k)$ be an automorphism and let $T_1, T_2, T_3$ be the tertians of $H_k$. Then, any tertian is mapped by $\sigma$ into another tertian as a whole, that is, for any $u, v \in T_i$, $\sigma(u), \sigma(v) \in T_j$, for $1 \le i, j \le 3$.
\end{proposition}
\begin{proof}
For a tournament $T$ and two vertices $x, y \in V(T)$, define
\[ D_T(x,y) = |\{ z \in T \mid zx \in A(T) \Leftrightarrow yz \in A(T) \}|. \]
That is, $D_T(x,y)$ is the number of vertices in $T$ having different relationships with $x$ and $y$. Now, suppose $u, v$ belong to the same tertian $T_i$ of $H_k$. Clearly, since all vertices outside $T_i$ have the same relationship with $u$ and $v$, only vertices in $T_i$ can have a different relationship with $u$ and $v$ and, then, $D_{H_k}(u,v) < 3^{k-1}$. However, if $\sigma(u)$ and $\sigma(v)$ belong to different tertians for an automorphism $\sigma \in Aut(H_k)$, then all the vertices in the other tertian will have a different relationship with $u$ and $v$ and, then, $D_{\sigma(H_k)}(\sigma(u),\sigma(v)) \ge 3^{k-1}$. Since an automorphism should preserve adjacencies, $\sigma$ cannot be an automorphism in $Aut(H_k)$ as we supposed. This contradiction shows that $\sigma(u)$ and $\sigma(v)$ must belong to the same tertian, say $T_j$ ($j$ being not necessarily different from $i$).  
\end{proof}

We consider labelings that play an important role in Section~\ref{sec:d_num}. In the discussion about 2-distinguishing labelings, although vertex and arc labels formally belong to the set $\{0,1\}$, from this point on we refer to label 1 as {\em white} and to label 2 as {\em black}. 

\begin{definition}\label{def:H}
We define black labelings and white labelings for $H_k$ as follows:
\begin{itemize} 
\item For $H_0$, a {\em black} {\em (white)} labeling consists of labeling the unique vertex of $H_0$ black (white).
\item For $H_k$, $k > 0$, a {\em black} {\em (white)} labeling contains two copies of $H_{k-1}$ with a black (white) labeling and one copy of $H_{k-1}$ with a white (black) labeling.
\end{itemize}
\end{definition}

\begin{figure}[h]
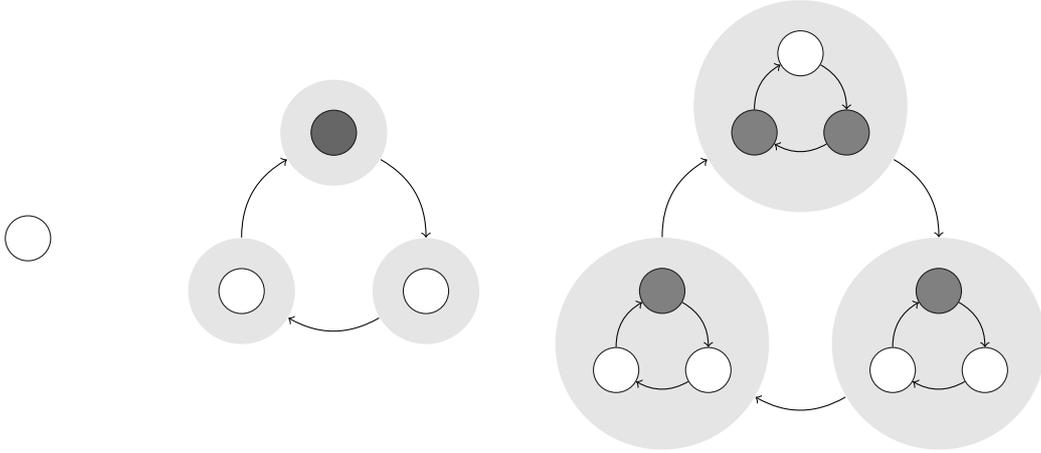

\whiteZERO \hglue 7mm \whiteONE \hglue 10mm \whiteTWO
\caption{From left to right, white labelings for tournaments $H_0$, $H_1$, and $H_2$. Tertians are shadowed in grey. One arc between two tertians implies all arcs between their respective nodes in the same direction.}
\end{figure}

\section{Small distinguishing vertex classes}\label{sec:d_num}

Just by observing that distinguishing vertex classes are closed by complementation, we obtain an upper bound for their size with the help of Gluck's theorem.

\begin{theorem}\label{th:dn}
For any tournament $T$ of order $n$, $\rho(T) \le \lfloor n/2 \rfloor$.
\end{theorem}
\begin{proof}
Let $T$ be a tournament of order $n$ and let $S$ be a distinguishing vertex set given by Gluck's Theorem. Then, the set $V(T) \setminus S$ is also distinguishing and either $S$ or $V(T) \setminus S$ has size at most $\lfloor n/2 \rfloor$.
\end{proof}

The following proposition shows that the bound given in Theorem~\ref{th:dn} is optimal for the family ${\cal H} = \{H_k\}_{k \ge 0}$ of tournaments from Definition~\ref{def:H}.

\begin{proposition}\label{pr:lower_dn}
For every $k \ge 0$, $\rho(H_k) \ge \lfloor 3^k/2 \rfloor$.
\end{proposition}
\begin{proof}
We use a more informative statement to prove the result.

\begin{claim}
For any 2-distinguishing vertex labeling $\phi$ of $H_k$:
\begin{enumerate}
\item $H_k$ has at least $\frac{3^k-1}{2}$ black vertices in $\phi$.
\item If $H_k$ has exactly $\frac{3^k-1}{2}$ black (white) vertices in $h$, then $h$ is a white (black) labeling for $H_k$.
\end{enumerate}
\end{claim}
\begin{proof}
We proceed by induction on $k$. If $k = 0$, both points 1 and 2 are trivially true. Suppose then that $k > 0$ and let $\phi$ be a 2-distinguishing vertex labeling for $H_k$. Then, since any automorphism in one of the three subtournaments of $H_k$ isomorphic to $H_{k-1}$ is also an automorphism of $H_k$, $\phi$ must be distinghishing for all three copies of $H_{k-1}$. By induction hypothesis, point 1 implies that any of the copies of $H_{k-1}$ must have at least $\frac{3^{k-1}-1}{2}$ black vertices in $h$. Therefore, $H_k$ must have at least 
\[ 3 \cdot \frac{3^{k-1}-1}{2} = \frac{3^k-3}{2} = \frac{3^k-1}{2}-1\]
black vertices in $h$. If $H_k$ contains exactly $\frac{3^k-1}{2}-1$ black vertices, the three copies of $H_{k-1}$ must contain exactly $\frac{3^{k-1}-1}{2}$ black vertices and, by induction hypothesis, point 2 implies that the restriction of $\phi$ to any of the copies of $H_{k-1}$ is a white labeling for it. Then, there is a nontrivial automorphism of $H_k$ consisting of a rotation of its subtournaments which respects the labeling $\phi$, which is a contradiction with the asumption that $\phi$ is distinghishing for $H_k$. Therefore, $H_k$ must contain at least $\frac{3^k-1}{2}$ black vertices, as we wanted to show in point 1.

As for point 2, if $H_k$ contains exactly $\frac{3^k-1}{2}$ black (white) vertices in $h$, since $\frac{3^k-1}{2} = 3 \cdot \frac{3^{k-1}-1}{2} + 1$ and all three copies of $H_{k-1}$ have at least $\frac{3^{k-1}-1}{2}$ black (white) vertices in $h$, it follows that two of the copies must have exactly $\frac{3^{k-1}-1}{2}$ black (white) vertices while the third one must have $\frac{3^{k-1}-1}{2}+1 = \frac{3^{k-1}+1}{2}$ black (white) vertices and $3^k - \frac{3^{k-1}+1}{2} = \frac{3^{k-1}-1}{2}$ white (black) vertices. By induction hypothesis, point 2 implies that two of the copies of $H_{k-1}$ have a white (black) labeling while the third one has a black (white) labeling. Consequently, $h$ is a white (black) labeling for $H_k$.
\end{proof}

The previous claim implies that for any 2-distinguishing vertex labeling $\phi$ of $H_k$, there are at least $\frac{3^k-1}{2} = \lfloor 3^k/2 \rfloor$ black vertices. Therefore, $\rho(H_k) \ge \lfloor 3^k/2 \rfloor$.
\end{proof}

\begin{theorem}\label{th:exact_bound}
For every $k \ge 0$, there is a tournament of order $n = 3^k$ such that $\rho(T) = \lfloor n/2 \rfloor$.
\end{theorem}
\begin{proof}
Given $k \ge 0$, we take $T = H_k$, which has order $n = 3^k$.
From Theorem~\ref{th:dn} and Proposition~\ref{pr:lower_dn}, $\rho(T) = \lfloor 3^k/2 \rfloor = \lfloor n/2 \rfloor$.
\end{proof}

\section{Small distinguishing arc classes}\label{sec:d_ind}

To get un upper bound of the cost of 2-distinghishing tournaments with the arcs, we will use the concept of determining set. Given a digraph $G$, a subset $S \subseteq V(G)$ is a {\em determining set} of $G$ if for any $\varphi, \psi \in Aut(G)$ such that $\varphi(x) = \psi(x)$ for all $x \in S$, $\varphi = \psi$. Thus, the action of an automorphism on $S$ determines its action on $V(G)$. From the group theory perspective, the pointwise stabilizer of a determining set is trivial while the setwise stabilizer of a distinguishing set is trivial (and therefore, every distinguishing set is a deterimining set).

The {\em determining number} of a digraph $G$, denoted by $Det(G)$, is defined as the minimum size of a determining set for $G$. We will use Theorem 8 in \cite{L}. 

\begin{theorem}\label{th:det}\cite{L}
For every order $n$ tournament $T$, $Det(T) \le \lfloor n/3 \rfloor$.
\end{theorem}

To get an upper bound for $\rho'(T)$, where $T$ is a tournament of order $n$, we start considering a determining set $S \subseteq V(T)$ that, according to Theorem~\ref{th:det}, can be selected with size bounded by $\lfloor n/3 \rfloor$. We can now singularize the vertices in $S$ by coloring some of the arcs in the subtournament of $T$ induced by $S$, $T[S]$. An easy way to do it is coloring in black the arcs of a Hamiltonian path in $T[S]$, and coloring the rest of the arcs in $T$ in white. This way, all the vertices in $S$ will be at a different distance (through the black arcs) from the beginning of the black path, and therefore, $S$ will be fixed pointwise, and $\rho'(T) \le \lfloor n/3 \rfloor -1$. However, we can push the upper bound down by combining determining sets with distinguishing sets.

\begin{theorem}\label{th:di}
For any order $n$ tournament $T$, $\rho'(T) \le \lfloor 7n/36 \rfloor + 3$.
\end{theorem}
\begin{proof}
Let $T$ be a tournament of order $n$. Then, by Theorem~\ref{th:det}, there exists a determining set $S \subseteq V(T)$ such that $|S| \le \lfloor n/3 \rfloor$. Consider the subtournament of $T$ induced by $S$, $T[S]$. By Theorem~\ref{th:dn}, $\rho(T[S]) \le \lfloor |S|/2 \rfloor \le \lfloor n/6 \rfloor$ and, therefore there exists a distinguishing set $R \subseteq S$ that proves it.


We will label some of the arcs in $S$ in black in such a way that every vertex in $S$ will be the extreme of some black arc. In the first place, we select the vertices in $S \setminus R$ by pairs and label the arcs joining the extremes of the selected pairs in black. Then, we select the vertices in $R$ by triples and, for each triple, we label two of its arcs in black. Note that the vertices from $S \setminus R$ which are incident to a black arc cannot be exchanged in any automorphism with the vertices in $R$ which are also incident to a black arc. In case $|S \setminus R|$ is not even or $|R|$ is not a multiple of 3, the previous method of grouping the vertices may leave at most 3 vertices which are not the extremes of any black arc, a maximum of one in $S \setminus R$ and two in $R$. Call $U = \{u_i \mid 1 \le i \le 3 \}$ to this set. 
For every possible cardinality of $U$, we calculate how many additional black arcs are needed to avoid the exchange of vertices from $S \setminus R$ and $R$ (and viceversa) in a nontrivial automorphism:
\begin{itemize}
\item $|U| = 0$. In this case, represented in Figure~\ref{fig:con}, all the vertices in $S \setminus R$ and $R$ are joined by black paths according to the above method. To complete the labeling, label all the remaining arcs in $T$ in white. Now, note that no vertex $u$ in $S \setminus R$ can map to a vertex in $R$ in an automorphism $\varphi$ because $u$ is the extreme of a black path of length 2 while $\varphi(u)$ is either the extreme of a black path of length 3 or its middle point. Therefore, vertices from the two parts of the partition of $S$ cannot be exchanged in any automorphism. Since $R$ is a distinguishing set for $S$,  the whole of $S$ lacks a nontrivial automorphism after the labeling, that is, $S$ is rigid. Additionally note that every vertex in $S$ is the extreme of some black arc while all vertices in $V(T) \setminus S$ are only the extremes of white arcs; therefore, no automorphism can map a vertex in $S$ to a vertex in $V(T) \setminus S$. Since $S$ is a determining set for $T$, our labeling is distinguishing. As for the size of the black label class, observe that there is one black arc for every 2 vertices in $S \setminus R$ and two black arcs for every 3 vertices in $R$. Since we know that $|R| \le \lfloor |S|/2 \rfloor$, we have at most
\begin{equation}\label{eq1}
\frac{|S \setminus R|}{2} + \frac{2 |R|}{3} \le 
\frac{3 (|S \setminus R|+|R|)+|R|}{6} =
\frac{|S|}{2} + \frac{|R|}{6} \le 
\frac{\lfloor n/3 \rfloor}{2} + \frac{\lfloor n/6 \rfloor}{6} \le 
\Bigl\lfloor \frac{7n}{36} \Bigr\rfloor
\end{equation}
black arcs.

\begin{figure}[t]
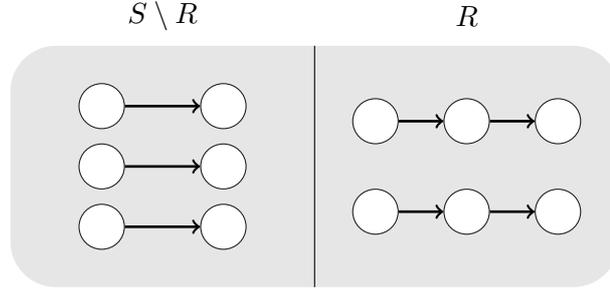
\label{fig:con}
\centerline{\conjunts}
\caption{Example of black arcs in $S$ for Theorem~\ref{th:di}, case $|U| = 0$.}
\end{figure}

\item $|U| = 1$. We complete our labeling by labeling an arc in $T[S]$ incident to $u_1$ in black in such a way that if we label the rest of the arcs in $T$ in white, all nontrivial automorphisms in $T[S]$ will be broken. To do so, consider the following subcases:
\begin{enumerate}
\item $R = \emptyset$. Then, $|S \setminus R| > 0$ and we label an arc from a vertex in $S \setminus R$ to $u_1$ in black.
\item $R \neq \emptyset$. Then, we label an arc from a vertex in $R$ to $u_1$ in black.
\end{enumerate}
Note that in both of the above subcases, $u_1$ is joined to a black path which is unlike the rest of black paths (the only one of length 2 in subcase 1, the only one of length 3 in subcase 2). Then, $u_1$ cannot be mapped to any other vertex in an automorphism in $T[S]$ and, similarly to the previous case ($|U| = 0$), we conclude that our labeling is distinguishing. As for the size of the black label class, we have 
\[ \Bigl\lfloor \frac{|S \setminus R|}{2}  \Bigr\rfloor + \Bigl\lfloor \frac{2 |R|}{3} \Bigr\rfloor + 1 \le
\Bigl\lfloor \frac{|S \setminus R|}{2} + \frac{2 |R|}{3} \Bigr\rfloor + 1 \le
\Bigl\lfloor \frac{7n}{36} \Bigr\rfloor + 1
\]
black arcs (where the last inequality derives from Equation~\ref{eq1}).

\item $|U| = 2$. Similarly to the previous case, we complete the labeling by labeling arcs in $T[S]$ incident to the vertices in $U$ in black, and labeling the rest of the arcs in $T$ in white. We consider two subcases:
\begin{enumerate}
\item $S \setminus R = \emptyset$. Then, we just label an arc joining $u_1$ to $u_2$ in black.
\item $S \setminus R \ne \emptyset$. Then, we label an arc from a vertex in $S \setminus R$ to $u_1$ and an arc joining $u_1$ to $u_2$ in black.
\end{enumerate}
Note that in both of the above subcases, $u_1$ and $u_2$ belong to a black path which is unlike the rest of black paths (the only one of length 1 in subcase 1, the only one of length 3 in subcase 2). Then, $u_1$ and $u_2$ cannot be mapped to any vertices in an automorphism in $T[S]$ and, similarly to the previous cases, we conclude that our labeling is distinguishing. The number of black arcs in our labeling can be obtained in a similar way to the previous case, being at most $\lfloor 7n/36 \rfloor + 2$ since here we may need to add two additional black arcs.

\item $|U| = 3$. Similarly to the two previous cases, we complete the labeling by labeling arcs in $T[S]$ incident to the vertices in $U$ in black, and labeling the rest of the arcs in $T$ in white. We consider two subcases:
\begin{enumerate}
\item $R = \emptyset$. Then, we label an arc joining $u_1$ to $u_2$ and an arc joining $u_2$ to $u_3$ in black. 
\item $R \neq \emptyset$. Then, as in the previous subcase, we color an arc joining $u_1$ to $u_2$ and an arc joining $u_2$ to $u_3$ in black. Additionally, we color an arc from a vertex in $R$ to $u_1$ in black.
\end{enumerate}
Note that in both of the above subcases, $u_1$, $u_2$, and $u_3$ belong to a black path which is unlike the rest of black paths (the only one of length 2 in subcase 1, the only one of length 4 in subcase 2). Then, $u_1$, $u_2$, and $u_3$ cannot be mapped to any vertices in an automorphism in $T[S]$ and, similarly to the previous cases, we conclude that our labeling is distinguishing. The number of black arcs in our labeling can be obtained in a similar way to the two previous cases, being at most $\lfloor 7n/36 \rfloor + 3$ since here we may need to add three additional black arcs.
\end{itemize}

Therefore, in all cases our labeling for $T$ is distinguishing and proves that $\rho'(T) \le \lfloor 7n/36 \rfloor + 3$.
\end{proof}

\medskip
We now show that the family $\{H_k\}_{k \ge 0}$ introduced in Section~\ref{sec:d_num} provides a lower bound for the distinguishing index of tournaments. We call {\em basic module} to any of the pairwise disjoint modules referred to in Observation~\ref{ob:disjoint}. Note that a nontrivial automorphism in any basic module trivially extends to $H_k$ as a consequence of the definition of module. This fact leads to the following lower bound for $\rho'(H_k)$.

\begin{proposition}\label{pr:lower_di}
For every $k \ge 1$, $\rho'(H_k) \ge \lceil 3^{k-1}/2 \rceil$.
\end{proposition}
\begin{proof}
Let $k \ge 1$. Since $\vec{C}_3$ is not rigid, all the $3^{k-1}$ basic modules of $H_k$ must contain an endpoint of some black arc if automorphisms in $H_k$ are to be broken. Since $3^{k-1}$ is odd, $\lfloor 3^{k-1}/2 \rfloor$ black arcs can have a maximum of $3^{k-1}-1$ endpoints, leaving at least one of the modules with a nontrivial automorphism. Therefore, $\rho'(H_k) \ge \lceil 3^{k-1}/2 \rceil$.
\end{proof}

We show that the bound $\lceil 3^{k-1}/2 \rceil$ is also un upper bound for the family of tournaments $\{ H_k \}_{k \ge 1}$.

\begin{proposition}\label{pr:upper_di}
For every $k \ge 1$, $\rho'(H_k) \le \lceil 3^{k-1}/2 \rceil$.
\end{proposition}
\begin{proof}
We call {\em primitive} any arc whose endpoints belong to the same basic module in $H_k$. We now refine the statement by considering primitive black arcs. 

\begin{figure}[t]
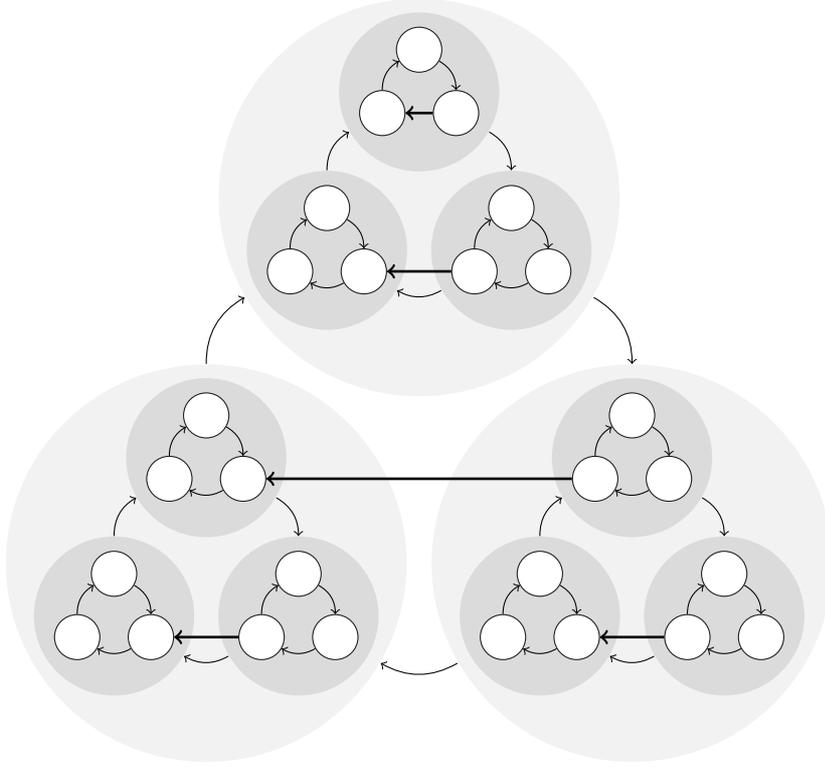
\label{fig:arc}
\centerline{\arclabelingtwo}
\caption{Arc labeling implied by Claim~\ref{cl:upper} for tournament $H_3$. The five straight thick arcs represent the only black arcs.}
\end{figure}

\begin{claim}\label{cl:upper}
For every $k \ge 2$, there is an 2-distinguishing arc labeling $\phi$ of $H_k$ with at most $\lceil 3^{k-1}/2 \rceil$ primitive black arcs.
\end{claim}
\begin{proof}
If $k = 2$, $H_k$ consists of three basic modules. Then, we define the labeling $\phi$ depicted in Figure~\ref{fig:arc} having a primitive black arc (the upper one in the figure) and a black arc going across the two remaining basic modues (the lower ones in the figure). Since each basic module contains vertices with unique properties and cannot be mapped into a different module in any automorphism, by Proposition~\ref{pr:module_move}, each tertian is mapped into itself. The fact that rotations inside the tertians are not possible either, $\phi$ is a 2-distinguishing arc labeling for $H_2$ satisfying the required conditions.

If $k > 2$, we know by induction hypothesis that $\lceil 3^{k-2}/2 \rceil$ black arcs are enough to break all nontrivial automorphisms in each of the three tertians. We also know that every tertian contains a primitive black arc. Consider now the labeling $\phi$ consisting in the union of the labelings given by induction hypothesis for the tertians with a single modification: we select two primitive black arcs from two tertians $T_i$ and $T_j$, $i \ne j$, we label the selected primitive black arcs white and then we label one  arc joining $T_i$ and $T_j$ black. There is still a primitive black arc in $H_k$ and, as a result of the relabeling, $\phi$ will have a maximum of
\[ 3 \Bigl\lceil \frac{3^{k-2}}{2} \Bigr\rceil - 2 + 1 = 3 \cdot \frac{3^{k-2} + 1}{2} - 1 
= \frac{3^{k-1} + 3 - 2}{2} = \Big\lceil \frac{3^{k-1}}{2} \Big\rceil. \]
black arcs as claimed. Furthermore, it is clear that labeling $\phi$ is 2-distinguishing for the tertians after the relabeling while, according to Proposition~\ref{pr:module_move}, an automorphism moving vertices between two different tertians would need to move all the vertices, but every tertian has properties different from the rest: for a first tertian (the upper one in the figure), there is no black arc connecting it to the other tertians, for a second tertian there is a black arc coming from outside (lower left), and for a third one there is a black arc going out (lower right). Therefore, $\phi$ is 2-distinguishing for $H_k$.
\end{proof}

Claim~\ref{cl:upper} proves the proposition for all $k \ge 2$. For $k = 1$, we can observe that tournament $H_1$ can be clearly made rigid by labeling one of its arcs in black and two of them in white. Therefore, for all $k \ge 1$, $\rho'(H_k) \le \lceil 3^{k-1}/2 \rceil$ as expected.  
\end{proof} 

We now get the following result.

\begin{theorem}
For every $k \ge 0$, there is a tournament $T$ of order $n = 3^k$ such that $\rho'(T) = \lceil n/6 \rceil$.
\end{theorem}
\begin{proof}
Given $k \ge 0$, we take $T = H_k$, which has order $n = 3^k$.
From Proposition~\ref{pr:lower_di} and Proposition~\ref{pr:upper_di}, $\rho'(T) = \lceil 3^{k-1}/2 \rceil = \lceil 3^k/6 \rceil = \lceil n/6 \rceil$.
\end{proof}

\section{Conclusions and open questions}\label{sec:con}

In~\cite{B1}, Boutin proves that  $\rho(Q_n) = {\cal O}(Det(Q_n))$, where $Q_n$ is the hypercube of dimension $n$, and asks in Question 9 whether this is also the case of other graph families. In relation with this question, she asks in Problem 4 of~\cite{B2} for graphs $G$ such that $\rho(G)$ is arbitrarily larger than $Det(G)$. We consider tournaments $T$ or order $n$ belonging to the family $\cal H$. From Section~\ref{sec:d_num} we have that $\rho(T) = \lfloor n/2 \rfloor$. On the other hand, it is clear that $Det(T) \le \lfloor n/3 \rfloor$ because each of the $n/3$ basic modules (isomorphic to $\vec{C}_3$ according to Observation~\ref{ob:disjoint}) needs to have either one or two black vertices in order to break the rotations. By Theorem~\ref{th:det}, we finally have $Det(T) = \lfloor n/3 \rfloor$. Therefore, $\rho(T)$ and $Det(T)$, for any $T \in {\cal H}$, are related by a factor of $3/2$ and we can answer affirmatively to both questions. 

\medskip
We conclude with a couple of open questions.

\begin{question}
Can the bound in Theorem~\ref{th:di} be improved? In particular, is $\rho'(T) \le \lceil \frac{n}{6} \rceil$ for any tournament $T$ of order $n$?
\end{question}

\begin{question}
What are the smallest values of $\rho(T)$ and $\rho'(T)$ for a cyclic tournament $T$?
\end{question}

\end{document}

%% file: graphs.tex
\def\conjunts{
\begin{tikzpicture}[scale=2,vertex_style/.style={circle,inner sep=6pt,draw=black!90,fill=black!50}]
  \pgfsetcornersarced{\pgfpoint{6mm}{6mm}}
  \pgfpathrectanglecorners{\pgfpointorigin}{\pgfpoint{4cm}{1.6cm}}
  \color{black!10}
  \pgfusepath{fill}
  \color{black}
  \pgfusepath{stroke}
  \draw (2,0) -- (2,1.6);
  \draw (0,0.4)
  node[vertex_style,fill=white!60] (1) at ++(0:0.6) {}
  node[vertex_style,fill=white!60] (2) at ++(0:1.4) {};
  \path (1) edge [line width = 1pt,->] (2);
  \draw (0,0.8)
  node[vertex_style,fill=white!60] (3) at ++(0:0.6) {}
  node[vertex_style,fill=white!60] (4) at ++(0:1.4) {};
  \path (3) edge [line width = 1pt,->] (4);
  \draw (0,1.2)
  node[vertex_style,fill=white!60] (5) at ++(0:0.6) {}
  node[vertex_style,fill=white!60] (6) at ++(0:1.4) {};
  \path (5) edge [line width = 1pt,->] (6);
  
  \draw (2,0.5)
  node[vertex_style,fill=white!60] (1p) at ++(0:0.4) {}
  node[vertex_style,fill=white!60] (2p) at ++(0:1) {}
  node[vertex_style,fill=white!60] (3p) at ++(0:1.6) {};
  \path (1p) edge [line width = 1pt,->] (2p);
  \path (2p) edge [line width = 1pt,->] (3p);
  \draw (2,1.1)
  node[vertex_style,fill=white!60] (4p) at ++(0:0.4) {}
  node[vertex_style,fill=white!60] (5p) at ++(0:1) {}
  node[vertex_style,fill=white!60] (6p) at ++(0:1.6) {};
  \path (4p) edge [line width = 1pt,->] (5p);
  \path (5p) edge [line width = 1pt,->] (6p);
  \node at (1,1.8) {$S \setminus R$};
  \node at (3,1.8) {$R$};
\end{tikzpicture}}

\def\arclabelingtwo{
\begin{tikzpicture}[scale=0.7,vertex_style/.style={circle,inner sep=6pt,draw=black!90,fill=black!50}]
\node[opacity=0.1,circle,fill=black!50,minimum size=150pt] (s0a) at ++(0,0) {};
\node[opacity=0.2,circle,fill=black!50,minimum size=60pt] (s1a) at +(90:2) {};
\node[opacity=0.2,circle,fill=black!50,minimum size=60pt] (s2a) at +(-30:2) {};
\node[opacity=0.2,circle,fill=black!50,minimum size=60pt] (s3a) at +(-150:2) {};
\draw (90:2)
node[vertex_style,fill=white!60] (1a) at ++(90:0.8) {}
node[vertex_style,fill=white!60] (2a) at ++(-30:0.8) {}
node[vertex_style,fill=white!60] (3a) at ++(-150:0.8) {};
\foreach \from/\to in {1a/2a,3a/1a,2a/3a} 
\path (\from) edge [bend left,->] (\to);
\draw (-30:2)
node[vertex_style,fill=white!60] (4a) at ++(90:0.8) {}
node[vertex_style,fill=white!60] (5a) at ++(-30:0.8) {}
node[vertex_style,fill=white!60] (6a) at ++(-150:0.8) {};
\foreach \from/\to in {4a/5a,5a/6a,6a/4a}
\path (\from) edge [bend left,->] (\to);
\draw (-150:2)
node[vertex_style,fill=white!60] (7a) at ++(90:0.8) {}
node[vertex_style,fill=white!60] (8a) at ++(-30:0.8) {}
node[vertex_style,fill=white!60] (9a) at ++(-150:0.8) {};
\foreach \from/\to in {7a/8a,8a/9a,9a/7a}
\path (\from) edge [bend left,->] (\to);
\path (s1a) edge [bend left,->] (s2a);
\path (s2a) edge [bend left,->] (s3a);
\path (s3a) edge [bend left,->] (s1a);
\path (6a) edge [line width = 1pt,->] (8a);

\draw ++(60:8)
node[opacity=0.1,circle,fill=black!50,minimum size=150pt] (s0b) at ++(0,0) {}
node[opacity=0.2,circle,fill=black!50,minimum size=60pt] (s1b) at ++(90:2) {}
node[opacity=0.2,circle,fill=black!50,minimum size=60pt] (s2b) at ++(-30:2) {}
node[opacity=0.2,circle,fill=black!50,minimum size=60pt] (s3b) at ++(-150:2) {};
\draw ++(60:8) ++(90:2)
node[vertex_style,fill=white!60] (1b) at ++(90:0.8) {}
node[vertex_style,fill=white!60] (2b) at ++(-30:0.8) {}
node[vertex_style,fill=white!60] (3b) at ++(-150:0.8) {};
\foreach \from/\to in {1b/2b,3b/1b} 
\path (\from) edge [bend left,->] (\to);
\path (2b) edge [line width = 1pt,->] (3b);
\draw ++(60:8) ++(-30:2)
node[vertex_style,fill=white!60] (4b) at ++(90:0.8) {}
node[vertex_style,fill=white!60] (5b) at ++(-30:0.8) {}
node[vertex_style,fill=white!60] (6b) at ++(-150:0.8) {};
\foreach \from/\to in {4b/5b,5b/6b,6b/4b}
\path (\from) edge [bend left,->] (\to);
\draw ++(60:8) ++(-150:2)
node[vertex_style,fill=white!60] (7b) at ++(90:0.8) {}
node[vertex_style,fill=white!60] (8b) at ++(-30:0.8) {}
node[vertex_style,fill=white!60] (9b) at ++(-150:0.8) {};
\foreach \from/\to in {7b/8b,8b/9b,9b/7b}
\path (\from) edge [bend left,->] (\to);
\path (s1b) edge [bend left,->] (s2b);
\path (s2b) edge [bend left,->] (s3b);
\path (s3b) edge [bend left,->] (s1b);
\path (6b) edge [line width = 1pt,->] (8b);

\draw ++(0:8)
node[opacity=0.1,circle,fill=black!50,minimum size=150pt] (s0c) at ++(0,0) {}
node[opacity=0.2,circle,fill=black!50,minimum size=60pt] (s1c) at ++(90:2) {}
node[opacity=0.2,circle,fill=black!50,minimum size=60pt] (s2c) at ++(-30:2) {}
node[opacity=0.2,circle,fill=black!50,minimum size=60pt] (s3c) at ++(-150:2) {};
\draw ++(0:8) ++(90:2)
node[vertex_style,fill=white!60] (1c) at ++(90:0.8) {}
node[vertex_style,fill=white!60] (2c) at ++(-30:0.8) {}
node[vertex_style,fill=white!60] (3c) at ++(-150:0.8) {};
\foreach \from/\to in {1c/2c,3c/1c,2c/3c} 
\path (\from) edge [bend left,->] (\to);
\draw ++(0:8) ++(-30:2)
node[vertex_style,fill=white!60] (4c) at ++(90:0.8) {}
node[vertex_style,fill=white!60] (5c) at ++(-30:0.8) {}
node[vertex_style,fill=white!60] (6c) at ++(-150:0.8) {};
\foreach \from/\to in {4c/5c,5c/6c,6c/4c}
\path (\from) edge [bend left,->] (\to);
\draw ++(0:8) ++(-150:2)
node[vertex_style,fill=white!60] (7c) at ++(90:0.8) {}
node[vertex_style,fill=white!60] (8c) at ++(-30:0.8) {}
node[vertex_style,fill=white!60] (9c) at ++(-150:0.8) {};
\foreach \from/\to in {7c/8c,8c/9c,9c/7c}
\path (\from) edge [bend left,->] (\to);
\path (s1c) edge [bend left,->] (s2c);
\path (s2c) edge [bend left,->] (s3c);
\path (s3c) edge [bend left,->] (s1c);
\path (6c) edge [line width = 1pt,->] (8c);
\path (3c) edge [line width = 1pt,->] (2a);
\foreach \from/\to in {s0a/s0b,s0b/s0c,s0c/s0a}
\path (\from) edge [bend left,->] (\to);
\end{tikzpicture}}

\def\whiteZERO{
\begin{tikzpicture}[scale=0.7,vertex_style/.style={circle,inner sep=6pt,draw=black!90,fill=black!60}]
\node[opacity=0,circle,fill=black!50,minimum size=80pt] (s1) at ++(0,-2) {};
\draw (90:0)
node[vertex_style,fill=white!60] (1) at ++(90:0) {};
\end{tikzpicture}}

\def\whiteONE{
\begin{tikzpicture}[scale=0.7,vertex_style/.style={circle,inner sep=6pt,draw=black!90,fill=black!60}]
\node[opacity=0,circle,fill=black!50,minimum size=80pt] (s1) at ++(0,-2) {};
\node[opacity=0.2,circle,fill=black!50,minimum size=40pt] (s1) at (90:2) {};
\node[opacity=0.2,circle,fill=black!50,minimum size=40pt] (s2) at (-30:2) {};
\node[opacity=0.2,circle,fill=black!50,minimum size=40pt] (s3) at (-150:2) {};
\draw ++(90:2)
node[vertex_style] (1) at ++(0:0) {};
\draw (-30:2)
node[vertex_style,fill=white!60] (4) at ++(0:0) {};
\draw (-150:2)
node[vertex_style,fill=white!60] (7) at ++(0:0) {};
\path (s1) edge [bend left,->] (s2);
\path (s2) edge [bend left,->] (s3);
\path (s3) edge [bend left,->] (s1);
\end{tikzpicture}}

\def\whiteTWO{
\begin{tikzpicture}[scale=0.7,vertex_style/.style={circle,inner sep=6pt,draw=black!90,fill=black!50}]
\node[opacity=0.2,circle,fill=black!50,minimum size=80pt] (s1) at (90:3) {};
\node[opacity=0.2,circle,fill=black!50,minimum size=80pt] (s2) at (-30:3) {};
\node[opacity=0.2,circle,fill=black!50,minimum size=80pt] (s3) at (-150:3) {};
\draw (90:3)
node[vertex_style,fill=white!60] (1) at ++(90:1) {}
node[vertex_style] (2) at ++(-30:1) {}
node[vertex_style] (3) at ++(-150:1) {};
\foreach \from/\to in {1/2,2/3,3/1}
\path (\from) edge [bend left,->] (\to);
\draw (-30:3)
node[vertex_style] (4) at ++(90:1) {}
node[vertex_style,fill=white!60] (5) at ++(-30:1) {}
node[vertex_style,fill=white!60] (6) at ++(-150:1) {};
\foreach \from/\to in {4/5,5/6,6/4}
\path (\from) edge [bend left,->] (\to);
\draw (-150:3)
node[vertex_style] (7) at ++(90:1) {}
node[vertex_style,fill=white!60] (8) at ++(-30:1) {}
node[vertex_style,fill=white!60] (9) at ++(-150:1) {};
\foreach \from/\to in {7/8,8/9,9/7}
\path (\from) edge [bend left,->] (\to);
\path (s1) edge [bend left,->] (s2);
\path (s2) edge [bend left,->] (s3);
\path (s3) edge [bend left,->] (s1);
\end{tikzpicture}}

\def\cycle6{
\begin{tikzpicture}[scale=0.8,inner sep=3pt]
\node[circle,draw=black!90,fill=yellow!60] (0) at (30:1.2) {2};
\node[circle,draw=black!90,fill=blue!50] (1) at (-30:1.2) {3};
\node[circle,draw=black!90,fill=blue!50] (2) at (-90:1.2) {4};
\node[circle,draw=black!90,fill=blue!50] (3) at (-150:1.2) {5};
\node[circle,draw=black!90,fill=blue!50] (4) at (-210:1.2) {6};
\node[circle,draw=black!90,fill=yellow!60] (5) at (-270:1.2) {1};
\foreach \from/\to in {0/1,1/2,2/3,3/4,4/5,5/0}
\path (\from) edge [-] (\to);
\end{tikzpicture}}

\def\dcycle6{
\begin{tikzpicture}[scale=0.8,inner sep=3pt]
\node[circle,draw=black!90,fill=blue!50] (0) at (30:1.2) {2};
\node[circle,draw=black!90,fill=blue!50] (1) at (-30:1.2) {3};
\node[circle,draw=black!90,fill=blue!50] (2) at (-90:1.2) {4};
\node[circle,draw=black!90,fill=blue!50] (3) at (-150:1.2) {5};
\node[circle,draw=black!90,fill=blue!50] (4) at (-210:1.2) {6};
\node[circle,draw=black!90,fill=yellow!60] (5) at (-270:1.2) {1};
\foreach \from/\to in {0/1,1/2,2/3,3/4,4/5,5/0}
\path (\from) edge [->] (\to);
\end{tikzpicture}}

\def\path6{
\begin{tikzpicture}[scale=0.8,inner sep=3pt]
\node[circle,draw=black!90,fill=blue!50] (0) at (30:1.2) {2};
\node[circle,draw=black!90,fill=blue!50] (1) at (-30:1.2) {3};
\node[circle,draw=black!90,fill=blue!50] (2) at (-90:1.2) {4};
\node[circle,draw=black!90,fill=blue!50] (3) at (-150:1.2) {5};
\node[circle,draw=black!90,fill=blue!50] (4) at (-210:1.2) {6};
\node[circle,draw=black!90,fill=yellow!60] (5) at (-270:1.2) {1};
\foreach \from/\to in {0/1,1/2,2/3,3/4,0/5}
\path (\from) edge [-] (\to);
\end{tikzpicture}}

\def\complete6{
\begin{tikzpicture}[scale=0.8,inner sep=3pt]
\node[circle,draw=black!90,fill=yellow!60] (0) at (30:1.2) {2};
\node[circle,draw=black!90,fill=yellow!60] (1) at (-30:1.2) {3};
\node[circle,draw=black!90,fill=yellow!60] (2) at (-90:1.2) {4};
\node[circle,draw=black!90,fill=yellow!60] (3) at (-150:1.2) {5};
\node[circle,draw=black!90,fill=blue!50] (4) at (-210:1.2) {6};
\node[circle,draw=black!90,fill=yellow!60] (5) at (-270:1.2) {1};
\foreach \from/\to in {0/1,0/2,0/3,0/4,0/5,1/2,1/3,1/4,1/5,2/3,2/4,2/5,3/4,3/5,4/5}
\path (\from) edge [-] (\to);
\end{tikzpicture}}

\def\Tournament5{
\begin{tikzpicture}[shorten >=1pt] 
\tikzstyle{vertex}=[circle,draw=black!90,fill=blue!50,minimum size=17pt,inner sep=0pt]
\foreach \name/\angle/\text in {Q-1/234/1, Q-2/162/2} 
\node[vertex,fill=yellow!60,xshift=6cm,yshift=.5cm] (\name) at (\angle:1cm) {$\text$};
\foreach \name/\angle/\text in {Q-3/90/3, Q-4/18/4, Q-5/-54/5} 
\node[vertex,xshift=6cm,yshift=.5cm] (\name) at (\angle:1cm) {$\text$};
\foreach \from/\to in {1/2,2/3,3/4,4/5,5/1,1/3,2/4,3/5,4/1,5/2} {\draw[->] (Q-\from) -- (Q-\to);}
\end{tikzpicture}}

\def\gtt4{
\begin{tikzpicture}[shorten >=1pt,->] 
$TT_4:$
\tikzstyle{vertex}=[circle,fill=black!25,minimum size=15pt,inner sep=0pt]
\foreach \name/\x in {1/1, 2/2, 3/3, 4/4} 
\node[vertex] (G-\name) at (\x,0) {$\name$};
\foreach \from/\to in {1/2,2/3,3/4} \draw (G-\from) -- (G-\to);
\draw (G-1) .. controls +(-30:1cm) and +(-150:1cm) .. (G-3);
\draw (G-1) .. controls +(-60:1cm) and +(-120:1cm) .. (G-4);
\draw (G-2) .. controls +(30:1cm) and +(150:1cm) .. (G-4);
\end{tikzpicture}}

\def\gtta4{
\begin{tikzpicture}[shorten >=1pt,->] 
$TT^*_4:$
\tikzstyle{vertex}=[circle,fill=black!25,minimum size=15pt,inner sep=0pt]
\foreach \name/\x in {1/1, 2/2, 3/3, 4/4} 
\node[vertex] (G-\name) at (\x,0) {$\name$};
\foreach \from/\to in {1/2,2/3,3/4} \draw (G-\from) -- (G-\to);
\draw (G-1) .. controls +(-30:1cm) and +(-150:1cm) .. (G-3);
\draw (G-4) .. controls +(-120:1cm) and +(-60:1cm) .. (G-1);
\draw (G-2) .. controls +(30:1cm) and +(150:1cm) .. (G-4);
\end{tikzpicture}}